\documentclass[a4paper,11pt]{amsart}
\usepackage{amssymb}
\usepackage{mathptmx}

\title[Discreteness of log discrepancies]{Discreteness of log discrepancies over log canonical triples on a fixed pair}

\author{Masayuki Kawakita}
\address{Research Institute for Mathematical Sciences, Kyoto University, Kyoto 606-8502, Japan}
\email{masayuki@kurims.kyoto-u.ac.jp}

\theoremstyle{plain}
\newtheorem{theorem}{Theorem}[section]
\newtheorem{proposition}[theorem]{Proposition}
\newtheorem{lemma}[theorem]{Lemma}
\newtheorem{corollary}[theorem]{Corollary}
\newtheorem{conjecture}[theorem]{Conjecture}

\theoremstyle{definition}
\newtheorem{definition}[theorem]{Definition}

\theoremstyle{remark}
\newtheorem{remark}[theorem]{Remark}
\newtheorem*{acknowledgements}{Acknowledgements}

\newcommand{\bA}{\mathbb{A}}
\newcommand{\bN}{\mathbb{N}}
\newcommand{\bQ}{\mathbb{Q}}
\newcommand{\bR}{\mathbb{R}}
\newcommand{\bZ}{\mathbb{Z}}
\newcommand{\cD}{\mathcal{D}}
\newcommand{\cG}{\mathcal{G}}
\newcommand{\cH}{\mathcal{H}}
\newcommand{\cO}{\mathcal{O}}
\newcommand{\cX}{\mathcal{X}}
\newcommand{\cZ}{\mathcal{Z}}
\newcommand{\fa}{\mathfrak{a}}
\newcommand{\fb}{\mathfrak{b}}
\newcommand{\fd}{\mathfrak{d}}
\newcommand{\fm}{\mathfrak{m}}
\newcommand{\fp}{\mathfrak{p}}

\DeclareMathOperator{\mld}{mld}
\DeclareMathOperator{\ord}{ord}
\DeclareMathOperator{\Spec}{Spec}

\begin{document}
\begin{abstract}
For a fixed pair and fixed exponents, we prove the discreteness of log discrepancies over all log canonical triples formed by attaching a product of ideals with given exponents.
\end{abstract}

\maketitle

\section{Introduction}
The log minimal model program (LMMP) is a program to find a good representative in each birational equivalence class of varieties by comparing the log canonical divisors. The log discrepancy, appearing in the relative log canonical divisor, is hence a fundamental invariant in the LMMP. For a triple $(X,\Delta,\fa)$, the \textit{log discrepancy} $a_E(X,\Delta,\fa)$ is attached to each divisor $E$ over $X$. The minimum of those $a_E(X,\Delta,\fa)$ with $E$ mapped onto a subset $Z$ of $X$ is called the \textit{minimal log discrepancy} at $\eta_Z$ and denoted by $\mld_{\eta_Z}(X,\Delta,\fa)$. Refer to Section \ref{sec:ld} for the precise definitions.

Shokurov conjectures the ascending chain condition (ACC) \cite{S88}, \cite[Conjecture 4.2]{S96} of the set of the minimal log discrepancies of all pairs with given coefficients in fixed dimension. Its importance is recognised in his reduction \cite{S04} of the termination of flips (in the relatively projective case) to this ACC and the lower semi-continuity of minimal log discrepancies. The main theorem of this paper is the \textit{discreteness} of log discrepancies over log canonical triples on a fixed pair. Let $\cD_X$ denote the set of divisors over $X$.

\begin{theorem}\label{thm:discrete}
Let $(X,\Delta)$ be a pair and $r_1,\ldots,r_k\in\bR_{\ge0}$. Then the set
\begin{align*}
\{a_E(X,\Delta,\prod_{j=1}^k\fa_j^{r_j})\mid\fa_j\subset\cO_X,\ E\in\cD_X,\ \textrm{$(X,\Delta,\prod_{j=1}^k\fa_j^{r_j})$ lc at $\eta_{c_X(E)}$}\}
\end{align*}
is discrete in $\bR$, where $c_X(E)$ is the centre of $E$ on $X$.
\end{theorem}

Note that it is trivial in the case of rational boundary and exponents. The condition of log canonicity is necessary, see Remark \ref{rmk:condition}.

Theorem \ref{thm:discrete} asserts a special case of Shokurov's ACC conjecture.

\begin{theorem}\label{thm:finite}
Let $(X,\Delta)$ be a pair and $r_1,\ldots,r_k\in\bR_{\ge0}$. Then the set
\begin{align*}
\{\mld_{\eta_Z}(X,\Delta,\prod_{j=1}^k\fa_j^{r_j})\mid\fa_j\subset\cO_X,\ Z\subset X\}
\end{align*}
is finite.
\end{theorem}

Theorem \ref{thm:finite} follows from Theorem \ref{thm:discrete} immediately since the minimal log discrepancies in Theorem \ref{thm:finite} are bounded from above by the maximum of $\mld_x(X,\Delta)$ for all $x\in X$.

We prove Theorem \ref{thm:discrete} in Section \ref{sec:discrete} by developing the study \cite{dFEM10}, \cite{dFEM11}, \cite{dFM09}, \cite{Kl08} of the ACC for \textit{log canonical thresholds} due to de Fernex, Ein, Musta\c{t}\u{a} and Koll\'ar. We use their construction of a \textit{generic limit} $(W,\fa)$, reviewed in Section \ref{sec:limit}, from a collection of bounded singularities $W_i$ and ideals $\fa_i$. $W$ is the spectrum of a complete local ring over an extension of the ground field. They showed the equivalence of the log canonicity of $(W,\fa)$ and general $(W_i,\fa_i)$ in order to obtain the ACC for log canonical thresholds on bounded singularities. We apply this equivalence to small perturbations of the exponents in $\fa$. It brings the boundedness of the orders appearing in the expression of $a_{F_i}(W_i,\fa_i)$, leading Theorem \ref{thm:discrete}.

Several extensions of Theorems \ref{thm:discrete}, \ref{thm:finite} and relevant remarks are given in Section \ref{sec:ext}. For example, applying to locally complete intersection (lci) singularities, we obtain the following in Corollary \ref{cor:lci}.

\begin{theorem}\label{thm:lci}
Fix an integer $d$ and $r_1,\ldots,r_k\in\bR_{\ge0}$. Then the set
\begin{align*}
\{\mld_{\eta_Z}(X,\prod_{j=1}^k\fa_j^{r_j})\mid\textrm{$X$ lci},\ \dim X\le d,\ \fa_j\subset\cO_X,\ Z\subset X\}
\end{align*}
is finite.
\end{theorem}

We work over an algebraically closed field $k$ of characteristic zero.
 
\section{Log discrepancies}\label{sec:ld}
A \textit{pair} $(X,\Delta)$ consists of a normal variety $X$ and a \textit{boundary} $\Delta$, that is, an effective $\bR$-divisor such that $K_X+\Delta$ is an $\bR$-Cartier $\bR$-divisor. We treat a \textit{triple} $(X,\Delta,\fa)$ by attaching a formal product $\fa=\prod_j\fa_j^{r_j}$ of finitely many coherent ideal sheaves $\fa_j$ with real exponents $r_j\in\bR_{\ge0}$. An \textit{extraction} of $X$ is a normal variety $X'$ with a proper birational morphism $\varphi\colon X'\to X$. A prime divisor $E$ on such an extraction $X'$ is called a divisor \textit{over} $X$, and the image $\varphi(E)$ on $X$ is called the \textit{centre} of $E$ on $X$ and denoted by $c_X(E)$. We denote by $\cD_X$ the set of divisors over $X$. We define the \textit{log discrepancy} of $E$ with respect to the triple $(X,\Delta,\fa)$ as
\begin{align*}
a_E(X,\Delta,\fa):=1+\ord_E(K_{X'}-\varphi^*(K_X+\Delta))-\ord_E\fa,
\end{align*}
where $\ord_E\fa:=\sum_jr_j\ord_E\fa_j$ for $\fa=\prod_j\fa_j^{r_j}$. The triple $(X,\Delta,\fa)$ is said to be \textit{log canonical} (\textit{lc}), \textit{Kawamata log terminal} (\textit{klt}) if $a_E(X,\Delta,\fa)\ge0$, $>0$ respectively for all $E\in\cD_X$, and said to be \textit{canonical}, \textit{terminal} if $a_E(X,\Delta,\fa)\ge1$, $>1$ respectively for all exceptional $E\in\cD_X$. Let $Z$ be an irreducible closed subset of $X$ and $\eta_Z$ its generic point. The \textit{minimal log discrepancy} $\mld_{\eta_Z}(X,\Delta,\fa)$ at $\eta_Z$ is the infimum of $a_E(X,\Delta,\fa)$ for all $E\in\cD_X$ with centre $Z$. It is either a non-negative real number or $-\infty$. The log canonicity of $(X,\Delta,\fa)$ about $\eta_Z$ is equivalent to $\mld_{\eta_Z}(X,\Delta,\fa)\ge0$. We say that $E\in\cD_X$ \textit{computes} $\mld_{\eta_Z}(X,\Delta,\fa)$ if $c_X(E)=Z$ and $a_E(X,\Delta,\fa)=\mld_{\eta_Z}(X,\Delta,\fa)$ (or negative when $\mld_{\eta_Z}(X,\Delta,\fa)=-\infty$). We are often reduced to the case when $Z$ is a closed point since $\mld_{\eta_Z}(X,\Delta,\fa)=\mld_z(X,\Delta,\fa)-\dim Z$ for general $z\in Z$, see \cite[Proposition 2.1]{A99}.

\section{Generic limits}\label{sec:limit}
The generic limit is a limit of ideals in a fixed local ring. It was constructed first by de Fernex and Musta\c{t}\u{a} in \cite{dFM09} using ultraproducts, and the construction was then simplified by Koll\'ar in \cite{Kl08}. It is clearly exposed in \cite[Section 4]{dFEM10}, \cite[Section 3]{dFEM11}.

Set $R=k[[x_1,\ldots,x_N]]$ with maximal ideal $\tilde\fm$. We fix integers $m$ and $k$. For every $l$, let $\cH_l$ be the Hilbert scheme parametrising ideals in $R$ containing $\tilde\fm^l$. Let $\cG$ be the parameter space for ideals in $R$ generated by polynomials in $\tilde\fm$ of degree $\le m$. Set $\cZ_l=\cG\times(\cH_l)^k$. We have a natural surjective map $t_l\colon\cZ_l\to\cZ_{l-1}$, and by generic flatness, there exists a stratification of $\cZ_l$ such that the restriction of $t_l$ on each stratum is a morphism.

We define a generic limit of the collection $\{(\tilde\fp_i;\tilde\fa_{i1},\ldots,\tilde\fa_{ik})\}_{i\in I}$ of $(k+1)$-tuples of ideals in $R$ indexed by an infinite set $I$, where $\tilde\fp_i$ are generated by polynomials in $\tilde\fm$ of degree $\le m$. One can construct locally closed irreducible subsets $Z_l^\circ\subset\cZ_l$ such that
\begin{enumerate}
\item
$t_l$ induces a dominant morphism $Z_l^\circ\to Z_{l-1}^\circ$,
\item
$I_l:=\{i\in I\mid(\tilde\fp_i;\tilde\fa_{i1}+\tilde\fm^l,\ldots,\tilde\fa_{ik}+\tilde\fm^l)\in Z_l^\circ\}$ is infinite,
\item
the set of points in $\cZ_l$ indexed by $I_l$ is dense in $Z_l^\circ$.
\end{enumerate}
We take the union $K=\bigcup_lk(Z_l^\circ)$ of the function fields by the inclusions $k(Z_{l-1}^\circ)\subset k(Z_l^\circ)$. For each $l$, the morphism $\Spec K\to Z_l^\circ\subset\cZ_l$ corresponds to a $(k+1)$-tuple $(\tilde\fp(l);\tilde\fa_1(l),\ldots,\tilde\fa_k(l))$ of ideals in $R_K=R\otimes_kK$. Then there exists a $(k+1)$-tuple $(\tilde\fp;\tilde\fa_1,\ldots,\tilde\fa_k)$ of ideals in $R_K$ such that $\tilde\fp(l)=\tilde\fp$ and $\tilde\fa_j(l)=\tilde\fa_j+\tilde\fm_K^l$, where $\tilde\fm_K=\tilde\fm R_K$. This $(\tilde\fp;\tilde\fa_1,\ldots,\tilde\fa_k)$ is a \textit{generic limit} of our collection $\{(\tilde\fp_i;\tilde\fa_{i1},\ldots,\tilde\fa_{ik})\}_{i\in I}$.

We set $W_i=\Spec R/\tilde\fp_i$, $W=\Spec R_K/\tilde\fp$ with closed points $o_i\in W_i$, $o\in W$, and $\fm_i=\tilde\fm/\tilde\fp_i$, $\fm=\tilde\fm_K/\tilde\fp$. We suppose $\tilde\fp_i\subset\tilde\fa_{ij}$, then $\tilde\fp\subset\tilde\fa_j$ and write $\fa_{ij}=\tilde\fa_{ij}/\tilde\fp_i$, $\fa_j=\tilde\fa_j/\tilde\fp$.

We compare log discrepancies over $W$ and $W_i$. The notions in Section \ref{sec:ld} are extended to the spectra of our complete local rings by the existence of their log resolutions due to Temkin in \cite{T08} after Hironaka. This extension is discussed in \cite{dFEM11}, \cite{dFM09} by de Fernex, Ein and Musta\c{t}\u{a}. The following proposition associates the minimal log discrepancy of the generic limit to those of the $\fm_i$-adic approximations of the original pairs. Proposition \ref{prp:limit} is a consequence of the basic fact that for a family of pairs, the minimal log discrepancy is constant on an open subfamily. The corresponding statement for log canonical thresholds is \cite[Proposition 4.4]{dFEM10} or \cite[Proposition 3.3]{dFEM11}. Our (\ref{itm:limit_computing}) is stronger, but it just needs the extra condition $l_E>\ord_E\fa_j$ for any $j$.

\begin{definition}
A subset $J_l$ of $I_l$ is said to be \textit{dense} if $J_l$ is infinite and the set of points in $\cZ_l$ indexed by $J_l$ is dense in $Z_l^\circ$.
\end{definition}

\begin{proposition}\label{prp:limit}
Suppose that $W_i$ has log terminal singularities. Then
\begin{enumerate}
\item
$W$ has log terminal singularities.
\item\label{itm:limit_mld}
For each $l$, there exists a dense subset $I_l^\circ$ of $I_l$ such that
\begin{align*}
\mld_o(W,\prod_j(\fa_j+\fm^l)^{r_j})=\mld_{o_i}(W_i,\prod_j(\fa_{ij}+\fm_i^l)^{r_j})
\end{align*}
for all $r_1,\ldots,r_k\in\bR_{\ge0}$ and all $i\in I_l^\circ$.
\item\label{itm:limit_computing}
Fix $r_1,\ldots,r_k\in\bR_{\ge0}$ and $E\in\cD_W$ computing $\mld_o(W,\prod_j\fa_j^{r_j})$. Then there exist an integer $l_E$, a dense subset $I_l^E$ of $I_l^\circ$ for each $l\ge l_E$, and $E_i\in\cD_{W_i}$ computing $\mld_{o_i}(W_i,\prod_j(\fa_{ij}+\fm_i^l)^{r_j})$ for each $i\in I_l^E$, such that
\begin{gather*}
\mld_o(W,\prod_j\fa_j^{r_j})=\mld_{o_i}(W_i,\prod_j(\fa_{ij}+\fm_i^l)^{r_j}),\\
\ord_E\fa_j=\ord_E(\fa_j+\fm^l)=\ord_{E_i}(\fa_{ij}+\fm_i^l)=\ord_{E_i}\fa_{ij}<l,
\end{gather*}
for all $i\in I_l^E$ with $l\ge l_E$.
\end{enumerate}
\end{proposition}

\begin{remark}\label{rmk:common}
One can choose $l_{E(s)}$ and $I_l^{E(s)}$ in Proposition \ref{prp:limit}(\ref{itm:limit_computing}) commonly for a finite collection $\{(r_1(s),\ldots,r_k(s);E(s))\}_s$.
\end{remark}

We shall use the effective ideal-adic semi-continuity of log canonicity due to Koll\'ar, and de Fernex, Ein and Musta\c{t}\u{a}. They applied it to the ACC for log canonical thresholds on bounded singularities.

\begin{theorem}[{\cite[Theorem 32]{Kl08}, \cite[Theorem 1.4]{dFEM10}}]
Let $W=\Spec\widehat{\cO_{X,x}}$ with closed point $o$ for some log canonical singularity $x\in X$, and $\fa_1,\ldots,\fa_k,\fb_1,\ldots,\fb_k\subset\cO_W$, $r_1,\ldots,r_k\in\bR_{\ge0}$. Suppose $\mld_o(W,\prod_j\fa_j^{r_j})=0$ and it is computed by $E\in\cD_W$. If $\fa_j+\fp_j=\fb_j+\fp_j$ for every $j$, where $\fp_j=\{f\in\cO_W\mid\ord_Ef>\ord_E\fa_j\}$, then $\mld_o(W,\prod_j\fb_j^{r_j})=0$.
\end{theorem}

\begin{corollary}\label{cor:lct}
If $\mld_o(W,\prod_j\fa_j^{r_j})=0$ in Proposition \ref{prp:limit}(\ref{itm:limit_computing}), then $\mld_{o_i}(W_i,\prod_j\fa_{ij}^{r_j})
\linebreak 
=0$ for $i\in I_l^E$.
\end{corollary}

\section{Discreteness}\label{sec:discrete}
The purpose of this section is to prove Theorem \ref{thm:discrete}. The theorem is reduced to the case when $X$ has $\bQ$-factorial terminal singularities by the existence of a $\bQ$-factorial terminal extraction $\varphi\colon X'\to X$ with $\Delta'\ge0$ for $K_{X'}+\Delta'=\varphi^*(K_X+\Delta)$, thanks to \cite{BCHM10}. Then we may assume $\Delta=0$ by forcing $\prod_j\fa_{ij}^{r_j}$ to absorb $\Delta$. Moreover, we may consider only the log discrepancies of divisors whose centres are closed points. Hence it suffices to prove the following theorem.

\begin{theorem}\label{thm:stability}
Let $X$ be a variety with log terminal singularities and $a,r_1,\ldots,r_k\in\bR_{\ge0}$. Let $I$ be an infinite set indexing $a_i=a_{E_i}(X,\prod_j\fa_{ij}^{r_j})\le a$ with $\fa_{i1},\ldots,\fa_{ik}\subset\cO_X$, $E_i\in\cD_X$ such that $x_i=c_X(E_i)$ is a closed point at which $(X,\prod_j\fa_{ij}^{r_j})$ is log canonical. Then there exists an infinite subset $I^\circ\subset I$ such that $a_i$ is constant for $i\in I^\circ$.
\end{theorem}

Since $X$ is covered by finitely many affine open subvarieties, we can fix integers $N$ and $m$ so that for each $i\in I$ there exists an ideal $\tilde\fp_i$ in $R=k[[x_1,\ldots,x_N]]$ generated by polynomials of degree $\le m$ which satisfies $\widehat\cO_{X,x_i}\simeq R/\tilde\fp_i$. We apply the generic limit construction in Section \ref{sec:limit} to the collection $\{(\tilde\fp_i;\tilde\fa_{i1},\ldots,\tilde\fa_{ik})\}_{i\in I}$; here we let $\fa_{ij}$ denote also the image in $R/\tilde\fp_i$ of $\fa_{ij}\widehat\cO_{X,x_i}$ by abuse of notation, and define $\tilde\fa_{ij}$ as the inverse image in $R$ of $\fa_{ij}$. The generic limit $(\tilde\fp;\tilde\fa_1,\ldots,\tilde\fa_k)$ is defined in some $R_K=R\otimes_kK$. We follow the notation in Section \ref{sec:limit}. We have $a_i=a_{F_i}(W_i,\prod_j\fa_{ij}^{r_j})$ for $F_i=E_i\times_XW_i$, and $(W_i,\prod_j\fa_{ij}^{r_j})$ is log canonical. By Proposition \ref{prp:limit}(\ref{itm:limit_mld}), $(W,\prod_j\fa_j^{r_j})$ is also log canonical.

We shall find perturbations of the exponents $r_j$ preserving the log canonicity. Set $r_0=1$. By permutation, we may assume that $r_0,\ldots,r_{k'}$ for some $0\le k'\le k$ form a basis of the $\bQ$-vector space spanned by $r_0,\ldots,r_k$. We write $r_j=\sum_{j'=0}^{k'}q_{jj'}r_{j'}$ with $q_{jj'}\in\bQ$, then $\prod_j\fa_j^{r_j}=\prod_{j'=0}^{k'}(\prod_j\fa_j^{q_{jj'}})^{r_{j'}}$ formally. We put $\fb_{j'}:=\prod_j\fa_{j}^{q_{jj'}}$ and $\fb_{ij'}:=\prod_j\fa_{ij}^{q_{jj'}}$. Setting $s_0=1$, for $\varepsilon>0$ we define the finite set
\begin{align*}
S_\varepsilon:=\{(s_1,\ldots,s_k)\mid s_j=\sum_{j'=0}^{k'}q_{jj'}s_{j'}\ \forall j,\ |s_{j'}-r_{j'}|=\varepsilon\ 1\le\forall j'\le k'\}.
\end{align*}

\begin{lemma}\label{lem:perturb}
There exist $\varepsilon$, $l$ and an dense subset $J_l\subset I_l$ such that all $s_j\ge0$ and $(W_i,\prod_j\fa_{ij}^{s_j})$ is log canonical for any $(s_j)\in S_\varepsilon$ and $i\in J_l$.
\end{lemma}

\begin{proof}
First we find $\varepsilon$ such that all $s_j\ge0$ and $(W,\prod_j\fa_j^{s_j})$ is log canonical for any $(s_j)\in S_\varepsilon$. Indeed, if $E\in\cD_W$ has $a_E(W,\prod_j\fa_j^{r_j})=0$, then $a_E(W)=\sum_{j'=0}^{k'}r_{j'}\ord_E\fb_{j'}$, and thus $\ord_E\fb_{j'}=0$ for $1\le j'\le k'$ by the $\bQ$-linear independence of $r_{j'}$. This means that the log discrepancy of $E$ remains zero when we perturb the exponents $r_1,\ldots,r_{k'}$ in $\prod_{j'}\fb_{j'}^{r_{j'}}(=\prod_j\fa_j^{r_j})$. Hence on a fixed log resolution of $(W,\prod_j\fa_j^{r_j})$, every divisor remains to have non-negative log discrepancy by sufficiently small such perturbation, which guarantees the existence of the required $\varepsilon$.

For each $s=(s_j)\in S_\varepsilon$, we fix $t_s\ge0$ such that $\mld_o(W,\prod_j\fa_j^{s_j}\fm^{t_s})=0$. By Corollary \ref{cor:lct} and Remark \ref{rmk:common}, we obtain $l$ and $J_l\subset I_l$ such that $\mld_{o_i}(W_i,\prod_j\fa_{ij}^{s_j}\fm_i^{t_s})=0$ for any $s\in S_\varepsilon$ and $i\in J_l$, meaning the log canonicity of $(W_i,\prod_j\fa_{ij}^{s_j})$.
\end{proof}

\begin{lemma}\label{lem:em}
$\sum_{j'=1}^{k'}|\ord_{F_i}\fb_{ij'}|\le \varepsilon^{-1}a_i$ for $i\in J_l$.
\end{lemma}

\begin{proof}
We choose $s_{ij'}=r_{j'}\pm\varepsilon$ for $1\le j'\le k'$ so that $\ord_{F_i}\fb_{ij'}/(s_{ij'}-r_{j'})\ge0$, and extend $(s_{ij'})$ to the $k$-tuple $(s_{ij})\in S_\varepsilon$. Then by Lemma \ref{lem:perturb}, $0\le a_{F_i}(W_i,\prod_j\fa_{ij}^{s_{ij}})=a_i-\varepsilon\sum_{j'=1}^{k'}|\ord_{F_i}\fb_{ij'}|$.
\end{proof}

Fix a positive integer $n$ such that $nK_X$ is a Cartier divisor and $nq_{jj'}\in\bZ$ for all $j,j'$. We define the finite set
\begin{align*}
A:=[0,a]\cap\Bigl(\frac{1}{n}\bZ+\bigl\{\sum_{j'=1}^{k'}r_{j'}m_{j'}\bigm|\sum_{j'=1}^{k'}|m_{j'}|\le\varepsilon^{-1}a,\ m_{j'}\in\frac{1}{n}\bZ\ \forall j'\bigr\}\Bigr).
\end{align*}
Then by Lemma \ref{lem:em}, $a_i=a_{F_i}(W_i,\prod_{j'=0}^{k'}\fb_{ij'}^{r_{j'}})=a_{F_i}(W_i,\fb_{i0})-\sum_{j'=1}^{k'}r_{j'}\ord_{F_i}\fb_{ij'}\in A$ for $i\in J_l$. Theorem \ref{thm:stability}, and Theorem \ref{thm:discrete}, are therefore completed.

\section{Extensions}\label{sec:ext}
First we remark the need of log canonicity in Theorem \ref{thm:discrete}.

\begin{remark}\label{rmk:condition}
Consider a non-lc pair $(\bA^2,(1+r)l)$ where $r>0$ is irrational and $l$ is a line. Let $E_1$ be the exceptional divisor of the blow-up of $\bA^2$ at a point on $l$, and define $E_p$ inductively as the exceptional divisor of the blow-up at the intersection of $E_{p-1}$ and the strict transform of $l$. Let $E_{p,0}=E_p$ and define $E_{p,q}$ inductively as the exceptional divisor of the blow-up at a general point on $E_{p,q-1}$. Then $a_{E_{p,q}}(\bA^2,(1+r)l)=q-pr$. The set of these log discrepancies is dense in $\bR$.
\end{remark}

The generic limit construction is applicable to bounded singularities in the sense \cite{dFEM11} of de Fernex, Ein and Musta\c{t}\u{a}. We say that a collection $\{x_i\in X_i\}_i$ of singularities is \textit{bounded} if there exist $m$ and $N$ such that for each $i$ there exists an ideal $\tilde\fp_i$ in $R=k[[x_1,\ldots,x_N]]$ generated by polynomials of degree $\le m$ which satisfies $\widehat\cO_{X_i,x_i}\simeq R/\tilde\fp_i$. Theorem \ref{thm:discrete} is formulated for such a collection.

\begin{theorem}
Let $\cX$ be a collection of varieties with bounded log terminal singularities, and $r_1,\ldots,r_k\in\bR_{\ge0}$. Then the set
\begin{align*}
\{a_E(X,\prod_{j=1}^k\fa_j^{r_j})\mid X\in\cX,\ \fa_j\subset\cO_X,\ E\in\cD_X,\ \textrm{$(X,\prod_{j=1}^k\fa_j^{r_j})$ lc at $\eta_{c_X(E)}$}\}
\end{align*}
is discrete.
\end{theorem}

We apply it to the minimal log discrepancies of Gorenstein singularities.

\begin{corollary}
Let $\cX$ be a collection of varieties with bounded normal Gorenstein singularities, and $r_1,\ldots,r_k\in\bR_{\ge0}$. Then the set
\begin{align*}
\{\mld_{\eta_Z}(X,\prod_{j=1}^k\fa_j^{r_j})\mid X\in\cX, \fa_j\subset\cO_X,\ Z\subset X\}
\end{align*}
is finite.
\end{corollary}

This follows from the boundedness \cite[Theorem 2.2]{K11} of the minimal log discrepancies of Gorenstein singularities with bounded embedding dimensions. Note that even for Gorenstein log canonical singularities, \cite[Theorem 2.2]{K11} holds by its proof, and Proposition \ref{prp:limit}(\ref{itm:limit_mld}), (\ref{itm:limit_computing}) hold since \cite[Appendix B]{dFEM11} is unnecessary.

We have a further application to lci singularities.

\begin{corollary}\label{cor:lci}
Fix an integer $d$ and $r_1,\ldots,r_k\in\bR_{\ge0}$. Then the set
\begin{align*}
\{a_E(X,\prod_{j=1}^k\fa_j^{r_j})\mid\textrm{$X$ lci},\ \dim X\le d,\ \fa_j\subset\cO_X,\ E\in\cD_X,\ \textrm{$(X,\prod_{j=1}^k\fa_j^{r_j})$ lc at $\eta_{c_X(E)}$}\}
\end{align*}
is discrete.
\end{corollary}

Corollary \ref{cor:lci}, and Theorem \ref{thm:lci}, follow from inversion of adjunction \cite{EM04} on lci varieties. We also use its consequence that an lci log canonical singularity of dimension $d$ has embedding dimension $\le2d$, see \cite[Proposition 6.3]{dFEM10}.

On the other hand, Shokurov's ACC conjecture is generalised to the case when the exponents vary in a fixed set satisfying the descending chain condition (DCC). Musta\c{t}\u{a} observed that an effective ideal-adic semi-continuity for  minimal log discrepancies implies the generalised ACC on a fixed pair (see \cite[Remark 1.5.1]{K10}). This semi-continuity is known in the klt case \cite[Theorem 1.6]{K10}, and for example, we can prove the following.

\begin{proposition}
Let $(X,\Delta)$ be a pair with rational $\Delta$, and $R$ a subset of $\bR_{\ge0}$ satisfying the DCC. Suppose that any accumulation point of $R$ is irrational. Then the set
\begin{align*}
\{\mld_{\eta_Z}(X,\Delta,\fa^r)\mid\fa\subset\cO_X,\ r\in R,\ Z\subset X\}
\end{align*}
satisfies the ACC.
\end{proposition}

\begin{proof}
As in the beginning of Section \ref{sec:discrete}, we are reduced to the case of terminal $X$, and we want the stability of any non-decreasing sequence of $a_i=\mld_{x_i}(X,\Delta,\fa_i^{r_i})\ge0$ with $r_i\in R$, $x_i$ closed point, where $i\in\bN$. By passing to a subsequence, we may assume that $\fa_i$ is non-trivial at $x_i$. Then $r_i$ are bounded by the maximum $b$ of $\mld_x(X,\Delta)$ for all $x\in X$, hence we may further assume that $\{r_i\}_i$ is a non-decreasing sequence which has a limit $r$. If $r\in\bQ$, then $r_i=r$ for large $i$ by the assumption on $R$, and the stability is trivial. Henceforth we assume $r\not\in\bQ$. As in Section \ref{sec:discrete}, we construct a generic limit $o\in W$, $\fd,\fa$ of $o_i\in W_i$, $\fd_i,\fa_i$ with $a_i=\mld_{o_i}(W_i,\fd_i\fa_i^{r_i})$, where $\fd_i$ is an ideal with fixed rational exponent, corresponding to $\Delta$. We fix $F_i\in\cD_{W_i}$ computing $\mld_{o_i}(W_i,\fd_i\fa_i^{r_i})$, that is,
\begin{align}\label{eqn:a_i}
a_i=a_{F_i}(W_i,\fd_i\fa_i^r)+(r-r_i)\ord_{F_i}\fa_i.
\end{align}

By Proposition \ref{prp:limit} or \cite[Corollary 3.4]{dFEM11}, $(W,\fd\fa^r)$ is log canonical. If $E\in\cD_W$ has $a_E(W,\fd\fa^r)=0$, then $\ord_E\fa=0$ by $r\not\in\bQ$. Thus we can find $t>0$ such that $(W,\fd\fa^{r+t})$ is log canonical as in the proof of Lemma \ref{lem:perturb}. We take $t'\ge0$ such that $\mld_o(W,\fd\fa^{r+t}\fm^{t'})=0$. Then Corollary \ref{cor:lct} shows $\mld_{o_i}(W_i,\fd_i\fa_i^{r+t}\fm_i^{t'})=0$ for infinitely many $i$. In particular,
\begin{align}\label{eqn:order}
t\ord_{F_i}\fa_i\le a_{F_i}(W_i,\fd_i\fa_i^r)\le a_i\le b.
\end{align}
Theorem \ref{thm:discrete} and (\ref{eqn:order}) imply the finiteness of possible choices for $a_{F_i}(W_i,\fd_i\fa_i^r)$ and $\ord_{F_i}\fa_i$ for such $i$. Hence they are constant for infinitely many $i$. Now (\ref{eqn:a_i}) provides the constancy of $a_i$ and $r_i$ for large $i$.
\end{proof}

\begin{remark}\label{rmk:mld_limit}
When $\fd_i=\cO_{W_i}$ and $r\not\in\bQ$ in the proof, one can further prove $a_i=a:=\mld_o(W,\fa^r)$ for large $i$. Indeed, we may assume $r_i=r$, and $a=\mld_{o_i}(W_i,(\fa_i+\fm_i^l)^r)\ge a_i$ for some $l>t^{-1}b$ by Proposition \ref{prp:limit}(\ref{itm:limit_computing}). Hence with (\ref{eqn:order}), we have $a\ge a_i=a_{F_i}(W_i,\fa_i^r)=a_{F_i}(W_i,(\fa_i+\fm_i^l)^r)\ge a$, meaning $a_i=a$.
\end{remark}

Remark \ref{rmk:mld_limit} is a special case of the following conjecture. The corresponding statement for log canonical thresholds is \cite[Corollary 3.4]{dFEM11}.

\begin{conjecture}\label{cnj:mld}
With the notation in Section \ref{sec:limit}, we suppose that $W_i$ has log terminal singularities and fix $r_1,\ldots,r_k\in\bR_{\ge0}$. Let
\begin{align*}
J:=\{i\in I\mid\mld_{o_i}(W_i,\prod_j\fa_{ij}^{r_j})=\mld_o(W,\prod_j\fa_j^{r_j})\}.
\end{align*}
Then $I_l\cap J$ is dense for each $l$.
\end{conjecture}

{\small
\begin{acknowledgements}
This paper has its origin in Musta\c{t}\u{a}'s problem on the ideal-adic semi-continuity for minimal log discrepancies. I am grateful to him for valuable discussions. Partial support was provided by Grant-in-Aid for Young Scientists (A) 24684003.
\end{acknowledgements}
}


\begin{thebibliography}{99}
\bibitem{A99}
F. Ambro,
On minimal log discrepancies,
Math.\ Res.\ Lett.\ \textbf{6} (1999), 573-580
\bibitem{BCHM10}
C. Birkar, P. Cascini, C. Hacon and J. McKernan,
Existence of minimal models for varieties of log general type,
J. Am.\ Math.\ Soc. \textbf{23} (2010), 405-468
\bibitem{dFEM10}
T. de Fernex, L. Ein and M. Musta\c{t}\u{a},
Shokurov's ACC conjecture for log canonical thresholds on smooth varieties,
Duke Math.\ J. \textbf{152} (2010), 93-114
\bibitem{dFEM11}
T. de Fernex, L. Ein and M. Musta\c{t}\u{a},
Log canonical thresholds on varieties with bounded singularities,
\textit{Classification of algebraic varieties},
EMS Series of Congress Reports (2011), 221-257
\bibitem{dFM09}
T. de Fernex and M. Musta\c{t}\u{a},
Limits of log canonical thresholds,
Ann.\ Sci.\ \'Ec.\ Norm.\ Sup\'er.\ (4) \textbf{42} (2009), 491-515
\bibitem{EM04}
L. Ein and M. Musta\c{t}\u{a},
Inversion of adjunction for local complete intersection varieties,
Am. J.\ Math.\ \textbf{126} (2004), 1355-1365
\bibitem{K11}
M. Kawakita,
Towards boundedness of minimal log discrepancies by Riemann--Roch theorem,
Am.\ J. Math.\ \textbf{133} (2011), 1299-1311
\bibitem{K10}
M. Kawakita,
Ideal-adic semi-continuity problem for minimal log discrepancies,
arXiv:
\linebreak 
1012.0395
\bibitem{Kl08}
J. Koll\'ar,
Which powers of holomorphic functions are integrable?,
arXiv:0805.0756
\bibitem{S88}
V. Shokurov,
Problems about Fano varieties,
\textit{Birational geometry of algebraic varieties, Open problems},
Katata 1988, 30-32
\bibitem{S96}
V. Shokurov,
$3$-fold log models,
J. Math.\ Sci. \textbf{81} (1996), 2667-2699
\bibitem{S04}
V. Shokurov,
Letters of a bi-rationalist V. Minimal log discrepancies and termination of log flips,
Tr.\ Mat.\ Inst.\ Steklova \textbf{246} (2004), 328-351,
translation in Proc.\ Steklov Inst.\ Math.\ \textbf{246} (2004), 315-336
\bibitem{T08}
M. Temkin,
Desingularization of quasi-excellent schemes in characteristic zero,
Adv.\ Math.\ \textbf{219} (2008), 488-522
\end{thebibliography}
\end{document}